\documentclass[preprint,12pt]{elsarticle}

\usepackage{amssymb}
\usepackage{amsmath}
\usepackage{amsfonts}
\usepackage{graphicx}
\usepackage{graphics}
\usepackage{tikz}
\usepackage{titlesec}

\newtheorem{theorem}{Theorem}

\newtheorem{corollary}[theorem]{Corollary}

\newtheorem{example}[theorem]{Example}

\newtheorem{lemma}[theorem]{Lemma}

\newtheorem{proposition}[theorem]{Proposition}
\newtheorem{remark}[theorem]{Remark}

\newenvironment{proof}[1][Proof]{\noindent \textbf{#1.} }{\  \rule{0.5em}{0.5em}}

\usepackage{tikz}
\usepackage{tkz-graph}

\journal{xxxxxxxxxxxxxxx}

\begin{document}
\begin{frontmatter}
\title{Lower bounds for Laplacian spread and relations with invariant parameters revisited}

\author{Enide Andrade}

\address{CIDMA-Center for Research and Development in Mathematics and Applications
         Departamento de Matem\'atica, Universidade de Aveiro, 3810-193, Aveiro, Portugal.}

\author{Maria Aguieiras A. de Freitas}

\address{Instituto de Matem\'{a}tica and COPPE/Produ\c{c}\~{a}o\\
 Universidade Federal do Rio de Janeiro, Rio de Janeiro, Brasil.}

\author{Mar\'{\i}a Robbiano, Jonnathan Rodr\'{\i}guez}

\address{Departamento de Matem\'{a}ticas, Universidad Cat\'{o}lica del Norte
         Av. Angamos 0610 Antofagasta, Chile.}

\begin{abstract}

Let $G=\left( V\left( G\right) ,E\left( G\right) \right) $ be an $\left(
n,m\right) $-graph and $X$ a nonempty proper subset of $V\left( G\right) $.
Let $X^{c}=V\left( G\right) \backslash X$.\ The \emph{edge density} of $X$ in
$G$ is given by
\begin{equation*}
\rho _{G}\left( X\right) =\frac{n\left\vert E_{X}\left( G\right) \right\vert
}{\left\vert X\right\vert \left\vert X^{c}\right\vert },
\end{equation*}%
where $E_{X}\left( G\right) \ $ is the set of edges in $G$ with one end in $%
X $ and the other in $X^{c}$. The Laplacian spread of a graph is the difference between the greatest Laplacian eigenvalue and the algebraic connectivity. In this paper,  we use the edge density of some nonempty
proper subsets of vertices in $G$ to establish new lower bounds for the Laplacian spread.
Also, using some known numerical inequalities some lower bounds for the Laplacian spread of a graph with a prescribed degree sequence are presented.
\end{abstract}


\end{frontmatter}

\section{Introduction and Motivation}

By an $(n,m)$-graph $G$ we mean an undirected simple graph with a vertex set
$V\left( G\right) $ of cardinality $n$ and an edge set $E\left( G\right) $
of cardinality $m$. If $e\in E(G)$ has end vertices $u$ and $v$ then they are neighbors and the edge is denoted by $uv$. Sometimes, if we label the vertices of $G$, we denote an edge by the labels of its end vertices, that is, if $e=v_i v_j$, then it can be simply denoted as $e= ij.$  A graph with no
edges (but at least one vertex) is called \emph{empty graph}. For $u\in V(G)$, the number of vertices adjacent to $u$ is denoted by $d\left( u\right) $,
and it is called the vertex degree of $u$. A \emph{pendent vertex}
is a vertex with degree one. The minimum and maximum vertex
degree of $G$ are denoted by $\delta (G)$  and $\Delta (G) $, respectively
(or simply $\delta$ and $\Delta $, respectively). A $q$-regular graph $G$ is a graph where every
vertex has degree $q$. The complete graph of order $n$ is an $\left(
n-1\right) $-regular graph with $n$ vertices and it is denoted by $K_{n}$.  The cycle with $n$ vertices is denoted by $C_{n}$.

Now, some interesting subgraphs of a graph are defined.
By a non trivial subset of vertices of $G$ we mean a nonempty proper subset of $V\left( G\right) $ and we will denote the induced subgraph with a non trivial vertex set $S\subset V\left( G\right) $ by $
\left\langle S\right\rangle|_{G} $. If $S\subseteq V\left( G\right) $ then $S^{c}$ is the set $V\left( G\right) \backslash S$.
A set of vertices that induces an
empty subgraph is called an \emph{independent set}. A \emph{matching} $N$ in a graph $G$ is a nonempty set of edges such that no
two have a vertex in common. The size of a matching is the number of edges
in it. A matching $N$ of $G$ is a \emph{dominating induced matching\ (DIM)}
of $G$ if every edge of $G$ is either in $N$ or has a common end vertex
with\ exactly one edge in $N$. A DIM is also called an \emph{efficient edge
domination set} (see for instance \cite{grinstead_et_al93}).\ It is important to notice that if $N$ is a DIM of $G$, then
there is a partition of $V(G)$ into two\ disjoint subsets $V(N)$ and $R$,
where $R$ is an independent set. Some literature concerning dominating induced matchings can be seen for instance in \cite{BM11, CL09, K09, LMRS}.

The complement, $\overline{G}$ of a graph $
G $ has the same vertex set as $G$, where vertices $u$ and $v$ are adjacent
in $\overline{G}$ if and only if they are non adjacent in $G$.

\noindent On the other hand, a non trivial vertex set $S\subseteq V(G)$ is $(\kappa
,\tau )$-regular set if $\left\langle S\right\rangle|_{G}$ is a $\kappa $-regular
subgraph in $G$ such that every vertex in $S^{c}$ has $\tau $ neighbors in $%
S $. This definition appeared at first place in \cite{Thompson} under the designation of eigengraphs and in \cite{Neumaier}, in the context of strongly regular graphs and designs.
More literature concerning this concept can be found for instance in \cite{Rama, CSZ2008}.

Considering two vertex disjoint graphs $G_{1}$ and $G_{2}$, the join of
$G_{1}$ and $G_{2}$ is the graph $G_{1}\vee G_{2}$ such that $V(G_{1}\vee
G_{2})=V(G_{1})\cup V(G_{2})$ and $E(G_{1}\vee G_{2})=E(G_{1})\cup
E(G_{2})\cup \{uv:u\in V(G_{1})$ and $v\in V(G_{2})\}$.

For a real symmetric matrix $M$, let us denote by $\theta _{i}\left(
M\right) $ the $i$-th largest eigenvalue of $M$. For the spectrum of a matrix $M$
(the multiset of the eigenvalues of $M$) we use the notation $\sigma (M)$. If $\theta $ is an eigenvalue of $M$ and $\mathbf{x}$ one of its eigenvectors, then the pair $(\theta ,\mathbf{x})$ is
an eigenpair of $M$.

As usual we denote the adjacency matrix of $G$ by $A(G)$ and the vertex degree matrix, $D(G)$ is the diagonal matrix
where its $v$-th diagonal entry is the degree of vertex $v,$ for $v \in V(G).$
From Ger\v{s}gorin Theorem, see \cite{varga}, the
matrix $L\left( G\right) =D\left( G\right) -A\left( G\right) $ is positive
semidefinite. This matrix is called the \emph{Laplacian matrix} of $G$  and its spectrum, $\sigma (L(G))$ is
called \emph{Laplacian
spectrum} of $G$. Let $\mathbf{e}$ be the
all ones vector, then $L\left( G\right) \mathbf{e}$ equals the null vector,
thus $0$ is always a Laplacian eigenvalue and its multiplicity is the number
of components of $G$.

The Laplacian eigenvalues, $\mu _{1}\geq \dots \geq \mu _{n-1}\geq \mu
_{n}=0 $ of $G$ and $\overline{\mu }_{1}\geq \dots \geq \overline{\mu }
_{n-1}\geq \overline{\mu }_{n}=0$ of $\overline{G}$ are related by
\begin{equation}
\overline{\mu }_{j}=n-\mu _{n-j},  \label{important}
\end{equation}%
for$\ 1\leq j\leq n-1.$ An important result in graph theory, see \cite
{Merris}, states that if $G$ is a graph with $n \geq 2$ vertices and has at least
one edge,$\ \Delta +1\leq \mu _{1}$, and for a connected graph the equality occurs if and only if $
\Delta =n-1$.
The \emph{algebraic connectivity} of a graph $G$ is defined as the second smallest Laplacian eigenvalue $\mu
_{n-1}$, \cite{Fiedler}, and it is sometimes denoted by \emph{a}$\left( G\right). $
Some important properties of this Laplacian eigenvalue can be found in \cite{Fiedler}. For instance, it is worth to recall that $G$ is connected if and only if $ \mu_{n-1}>0.$
There are some written surveys about algebraic connectivity, see \cite{Merris, Merris2, Mohar} and for applications see for instance,
\cite{Perron, lapl2}.

It is known, see \cite{Fiedler}, that if $G$ is a non complete graph, then
\textit{a}$\left( G\right) \leq \kappa _{0}(G)$, where $\kappa _{0}(G)$
denotes the vertex connectivity of $G$ (that is, the minimum number of
vertices whose removal yields a disconnected graph). Since, $\kappa
_{0}\left( G\right) \leq \delta (G)$, it follows that \textit{a}$\left(
G\right) \leq \delta(G) .$ The graphs for which the algebraic connectivity
attains the vertex connectivity are characterized in \cite{kirkland_et_al_02}%
.\

We introduce now  the \textit{edge density} of a
nonempty proper subset of vertices in $ G$. This concept appeared in \cite{Fallat_Kirk} although there was a previous definition given by Mohar in \cite{Mohar} that did not include the factor $n$ in his definition but used in
\cite{Fallat_Kirk} for author's purposes.

Let $G$ be an $\left( n,m\right) $-graph and $X$ a non trivial subset of
vertices of $G$.\ The \emph{edge density} of $X$ in $G$ is given by
\begin{equation*}
\rho _{G}\left( X\right) =\frac{n\left\vert E_{X}\left( G\right) \right\vert
}{\left\vert X\right\vert \left\vert X^{c}\right\vert }
\end{equation*}%
where $E_{X}\left( G\right) $ is the set of edges of $G$ with one end in $X$
and the other in $X^{c}.\ $Note that $\rho _{G}\left( X\right) =\rho
_{G}\left( X^{c}\right) $.\

The next two concepts are also important in this work.
The \emph{spread} of an $n\times n$ complex Hermitian matrix $M$
with eigenvalues $\theta _{1},\ldots ,\theta _{n}$ is defined by
\begin{equation*}
S\left( M\right) =\max_{i,j}\left\vert \theta _{i}-\theta _{j}\right\vert ,
\end{equation*}%
where the maximum is taken over all pairs of eigenvalues of $M$. There are
several papers devoted to this matrix parameter, see for instance \cite%
{Jonhson et al, Mirsky}. Attending to previous concept and in view that the smallest
Laplacian eigenvalue is zero, a nontrivial definition for the Laplacian
spread of a graph $G$ is the difference among the largest and the
second smallest Laplacian eigenvalue of $G$, see \cite{Fan}. So, the Laplacian spread,
denoted by $S_{L}\left( G\right) ,$ is given by%
\begin{equation*}
S_{L}\left( G\right) =\mu _{1}-a\left( G\right) .
\end{equation*}%
Note that if $G$ is not connected then $S_{L}(G)=\mu_{1}.$ Moreover,
since $\kappa _{0}(G)\leq \delta (G)$, it follows that $\mathit{a}\left(
G\right) \leq \delta (G)$ and then
\begin{equation}
S_{L}(G)\geq \Delta (G)+1-\delta\left( G\right),
\label{basic_lower_bound}
\end{equation}%
see e. g. \cite{Liu}. It is also immediate from (\ref{important}) that the
Laplacian spread of $G$ and $\overline{G}$ coincide. In consequence, if $\overline{G}$ is not connected,
then  $S_{L}(G) =S_{L}(\overline{G})=n-a(G).$

Recently, in \cite{Andrade}, by combining old techniques of interlacing eigenvalues and rank $1$ perturbation matrices, some lower bounds on the Laplacian spread were given.  Some of these bounds involve invariant parameters of graphs,
as it is the case of the bandwidth, independence number and vertex connectivity.

In this work, considering the referred work in  \cite{Andrade} as motivation, we search for new lower bounds on the Laplacian spread using the edge density of some nonempty proper subsets of vertices in $G$.  Firstly, we are concerned with those graphs
which contain special subgraphs such as empty subgraphs, DIM subgraphs or
special subset of vertices such as $\left( \kappa ,\tau \right) $-regular
subset of vertices. Then, using some known numerical inequalities we present some lower bounds for the Laplacian spread of a graph with a prescribed degree sequence.

The paper is organized in 4 sections. After Introduction, at Section 2, we present a known inequality due to Mohar, \cite{Mohar}, that relates the algebraic connectivity and the edge density. This result gives  a tool to establish some lower bounds for the Laplacian spread of a connected graph that are related with the edge density. Then, using these results we study lower bounds for the Laplacian spread of graphs that have a particular nontrivial subset of vertices, namely for graphs that have an independent nontrivial subset of vertices and a $(\kappa, \tau)$-regular subset of vertices. At Section 3, using some known numerical inequalities, additional lower bounds for the Laplacian spread of graphs are presented, this time using the degree sequence and another invariant parameter of the graph which depends of its Laplacian spectrum. Finally, in the last section we study an example in order to approach the Laplacian spread of a given graph, using  different parameters, such as algebraic connectivity, edge density of a set $X$  and minimum vertex degree. Some relations are observed.

\section{Edge density of a subset: relations with Laplacian spread}

In \cite{Fallat_Kirk}, the following inequality (previously presented in \cite{Mohar}), that relates the algebraic connectivity and the edge density was proved. The authors also studied the graphs $G$ for which the equality holds.

\begin{proposition}
\label{Lem1} \cite{Fallat_Kirk} Let $G$ be an $\left( n,m\right) $-graph. For
a non trivial subset $X$ of $V$, the edge density of $X$ satisfies%
\begin{equation*}
\label{firstineq}
\mathit{a}\left( G\right) \leq \rho _{G}\left( X\right) \leq \mu_1(G).
\end{equation*}%
Moreover, if a graph $G$ satisfies one of the equalities for some cut $%
E_{X}\left( G\right) $ (set of edges in $G$ whose removal yields a
disconnected graph), then there are integers $s$ and $t$ such that the
following conditions must hold:

\begin{enumerate}
\item Each vertex in $X$ is adjacent to $s$ vertices in $X^{c}$ and
each vertex in $X^{c}$ is adjacent to $t$ vertices in $X$, and

\item $s\left\vert X\right\vert =t\left\vert X^{c}\right\vert .$
\end{enumerate}
\end{proposition}
Note that, the condition 2. results from the fact that in $A(G)$,
after labeling the vertices of $X$,  the matrix $A(G)$ is partitioned into blocks and the blocks corresponding to the positions $(1,2)$ and $(2,1)$ have the same number of entries equal to $1$.

\begin{remark}
There are graphs $G$ with
subset of vertices $X$ such that $\rho _{G}\left( X\right) =\mathit{a}\left( G\right) $. Related with this fact, in \cite{Fallat_Kirk} the following
example was given: For a graph $G=\left( V,E\right) $ where $V=X\cup X^{c}$,
with $X$ and $X^{c}$ having $n_{1}$ and $n_{2}$ vertices,
respectively and satisfying $1.$ and $2.$ of Proposition \ref{Lem1} and such that
the induced subgraph $\left\langle X\right\rangle|_{G} =K_{n_{1}}$ and $
\left\langle X^{c}\right\rangle |_{G}=K_{n_{2}}$, then \textrm{a}$\left(
G\right) =\rho _{G}\left( X\right) .$
\end{remark}

As an immediate consequence of Proposition \ref{Lem1}, the next result gives a lower bound for $S_{L}(G)$ in function of edge density of a non trivial subset of the vertex set of $G$.

\begin{corollary}
\label{Theorem1}Let $G=\left( V,E\right) $ be a connected graph of order $n.$ Let $X$
a non trivial subset of $V$
and consider $\phi $ such that $a(G) \leq \phi \leq \mu _{1}.$ Then%
\begin{equation}
\label{first1}
S_{L}\left( G\right) \geq \left\vert \phi -\rho _{G}\left( X\right)
\right\vert .
\end{equation}%
If the equality holds then the conditions 1. and 2. of Proposition \ref{Lem1} hold for $X .$
Moreover, $\left\{ \rho _{G}\left(
X\right) ,\phi \right\} =\left\{ \mu _{1},\mathit{a}\left( G\right)
\right\} .$
\end{corollary}

\begin{remark}
Let $G=G_{1}\vee G_{2}$, let us denote by $n_{i}$ and $\mu _{j}\left(
G_{i}\right) $ the order and the $j$-$th$ largest eigenvalue of $G_{i},$ $i\in \{1,2\}$, respectively%
.\ Then, taking into account  \cite{ACMR}, \newline
${\small \sigma }_{L}\left( {\small G}\right) {\small =}\left\{ {\small 0,n}%
_{1}{\small +n}_{2}{\small ,n}_{2}{\small +\mu }_{j}\left( {\small G}%
_{1}\right) {\small ,n}_{1}{\small +\mu }_{i}\left( {\small G}_{2}\right)
{\small :1\leq j\leq n}_{1}{\small -1,}\text{\thinspace }{\small 1\leq i\leq
n}_{2}{\small -1}\right\} $, \newline
Taking $\phi =\mathit{a}\left( G\right) =\min \left\{ n_{2}+\mathit{a}%
\left( G_{1}\right) ,n_{1}+\mathit{a}\left( G_{2}\right) \right\} $ and
as $\rho _{G}\left( V\left( G_{1}\right) \right) =n_{1}+n_{2}$, with this choosen $\phi,$
one can see that the graph $G$ satisfies the equality in (\ref{first1}).
\end{remark}

Given a graph $G$ and $u\in V\left( G\right) $,  the set of neighbors of $u$ is denoted by
$N\left(u\right)$.
For a non trivial subset of vertices $X$, denote $N_{X}\left( u\right) :=X\cap N\left( u\right) $ and its cardinality
by $d_{X}\left( u\right) .$

\begin{remark}
\label{Important}
Let $G=\left( V,E\right) $ be a connected graph of order $n$. Let $%
X $ be a non trivial subset of $V$ with $n_{1}$ vertices. We set $n_{2}= n-n_{1}.$ 
Set  $\alpha_{X} =\frac{1}{n_{1}}\sum_{v\in X}d\left(
v\right) $ and $m_{X}=\frac{1}{2}\sum_{v\in X}d_{X}\left( v\right) $. 
Remark that
\begin{equation*}
\left( \alpha_{X} - \frac{2m_{X}}{n_1}\right) \left( 1+\tfrac{n_1}{n_{2}}\right) =\tfrac{n}{n_{2}}%
\left( \alpha_{X} - \frac{2m_{X}}{n_1}\right) =\tfrac{n}{n_{1}n_{2}}\sum_{v\in X}d_{X^{c}}\left(
v\right) =\rho _{G}\left( X\right).
\end{equation*}%

By direct calculation, it results
\begin{equation*}
\rho _{G}\left( X\right) =\left( \alpha _{X}-\frac{2m_{X}}{n_{1}}\right)
+\left( \alpha _{X^{c}}-\frac{2m_{X^{c}}}{n -n_{1}}\right) .
\end{equation*}
\end{remark}

\begin{corollary}
\label{new}Let $G=\left( V,E\right) $ be a connected graph of order $n$. Let $X$ be a non trivial subset of vertices of $G$ and let us
consider $\alpha _{X}$ as in Remark \ref{Important}. Then
\begin{equation}
S_{L}\left( G\right) \geq \max \left\{ \left\vert \rho _{G}\left(
X\right) -\alpha _{X}\right\vert ,\left\vert \rho _{G}\left(
X\right) -\alpha _{X^{c}}\right\vert \right\} .  \label{geometric}
\end{equation}
\end{corollary}

\begin{proof}
We set $|X|=n_{1}$. Note that $a(G) \leq \frac{1}{n_{1}}n_{1}\delta \leq \alpha _{X}\leq
\frac{1}{n_{1}}n_{1}\Delta \leq \mu _{1}$ and with the same argument we
have\ $a(G) \leq \alpha _{X^{c}}\leq \mu _{1}.\ $ 
 Thus, the inequality in (\ref{geometric}) follows from Corollary \ref{Theorem1}.\bigskip
\end{proof}

\subsection{Lower bound for the Laplacian spread of graphs with an independent nontrivial subset of vertices}
The following results and their applications use Corollary \ref{Theorem1} in some special cases. In this section we present a lower bound for the Laplacian spread of graphs that have an independent nontrivial subset of vertices.

\begin{corollary}
\label{independence} Let $G$ be an $(n,m)$-connected graph and $
T$ an independent set of $
G $ on $n_{1}$ vertices. Then%
\begin{equation}
S_{L}\left( G\right) \geq \frac{n\sum_{v \in T} d\left( v \right) }{n_{1}\left(
n-n_{1}\right) }-\delta \left( G\right) .  \label{lowindependent1}
\end{equation}%
If the equality holds, then there exist $s$ and $t$ such that $%
d_{T^{c}}\left( v\right) =d\left( v\right) =s$, for all $v\in T$ and $%
d_{T}\left( v\right) =t$, for all $v\in T^{c}$, with $sn_{1}=t(n-n_{1})$. Moreover $\mu _{1}\left(
G\right) =\frac{n\sum_{v \in T} d\left( v \right) }{n_{1}\left( n-n_{1}\right) }$
and $\mathit{a}\left( G\right) =\delta \left( G\right) .$
\end{corollary}

\begin{proof}
Note that if $T\subset V(G)$ is an independent set, $\left\vert E_{T}\left( G\right) \right\vert = \sum_{v \in T} d\left( v \right) \geq n_{1}\delta \left( G\right)$ and so $\rho _{G}\left( T\right) \geq \frac{n\delta \left( G\right)}{n-n_{1}} > \delta \left( G\right)$. The result follows from Corollary \ref{Theorem1}
\end{proof}

\begin{corollary}
\label{independence copy(1)} Let $G$ be an $(n,m)$-connected graph and
$X $ a nontrivial subset of vertices such that $\left\langle X\right\rangle| _{G}=K_{n_{1}}$
.\ Then%
\begin{equation*}
S_{L}\left( G\right) \geq \frac{n_{1}\left( n-1-\Delta\left( G\right) \right) }{n-n_{1}}.
\end{equation*}
\end{corollary}

\begin{proof}
Throughout the proof  $\Delta =\Delta \left( G\right) $. It is clear that $\left\langle X\right\rangle| _{\overline{G}}=\overline{K}_{n_{1}}$ and therefore it is an independent set in  $\overline{G}$. We now apply Corollary \ref{Theorem1} 
to $\overline{G}$ in order to obtain
\begin{equation}
S_{L}\left( \overline{G}\right) \geq \rho _{\overline{G}}\left( X\right)-\delta \left( \overline{G}\right).
\label{complement}
\end{equation}%

Note that $\left\vert E_{X}\left( \overline{G}\right) \right\vert = n_1(n-n_1) -\left\vert E_{X}\left(G\right) \right\vert$ and so $\rho _{\overline{G}}\left( X\right)= n-\rho _{G}\left( X\right)$.

Since $\delta \left( \overline{G}\right)= n-1-\Delta,$  $\left\vert E_{X}\left( G\right) \right\vert \leq n_1(\Delta - (n_1 - 1))$, for $\left\langle X\right\rangle| _{G}=K_{n_{1}}$, and recalling that $S_{L}\left( \overline{G}\right)
=S_{L}\left( G\right) $
we obtain

$$S_{L}\left( G\right) \geq \Delta  +1 - \rho _{\overline{G}}\left( X\right) \geq   \Delta +1 - \frac{nn_1 \left(\Delta  +1 - n_1\right)}{n_1 \left(n-n_1\right)} = \frac{n_{1}\left( n-1-\Delta\right) }{n-n_{1}}.$$
\end{proof}

\begin{remark}
If $G$ is a $p$-regular graph with $p>0$, and a nontrivial subset $X\ $of
vertices such that $\left\langle X\right\rangle| _{G}=K_{n_{1}},$ the lower
bound in Corollary \ref{independence copy(1)} is an improvement of
the lower bound in (\ref{basic_lower_bound}).
\end{remark}

\begin{remark}
Let $G$ be an $\left( n,m\right) $ connected graph and $T=\left\{
v_{1},v_{2},\ldots ,v_{n_{1}}\right\} $ the set of pendant vertices of $G$.$\ $%
Note that $\left\vert E_{T}\left( G\right) \right\vert = n_1$ and $\delta \left( G\right) =1$.
Then by Corollary~\ref{independence}, we may conclude that
\begin{equation}
S_{L}\left( G\right) \geq \frac{n}{n-n_{1}}-1=\frac{n_{1}}{n-n_{1}}.
\label{marca_xyz}
\end{equation}
which corresponds to the proportion among the pendant vertices and the non
pendent vertices in $G$. Moreover, if the equality holds, then there exists $t$ such that $d_{T}\left( v\right) =t$, for all $v\in T^{c}$ and $n_{1} 1=(n-n_{1})t$. Moreover, $\mu
_{1}\left( G\right) =\frac{n}{n-n_{1}}$ and $\mathit{a}\left( G\right) =1.$
As $a(G)\leq d(v)$, for all vertex $v\in V(G)$, the condition $a(G) =1$ implies that $n_{1}\geq n-n_{1}$ and this means that, for the equality case we must have $n_{1}\geq n/2$.
\end{remark}

Now, we search for a lower bound for $S_{L}(G)$ in terms of the order of $N,$ where $N$ is a DIM of $%
G.$

\begin{corollary}
\label{DIM2}Let $G$ be an $\left( n,m\right) $ connected graph.
Let us consider that $N=\left\{ \ i_{1}j_{1},\ldots ,i_{k}j_{k}\right\} $ is a
DIM of $G$ 
. Then
\begin{equation}
S_{L}(G)\geq \frac{n\left( m-k\right) }{2k\left( n-2k\right) }-\delta .
\label{ineqdim}
\end{equation}%
If equality holds, there are integer $s$ and $t$ such that $d_{R}(v) =s$, for all $v \in V(N)$ and $d_{V(N)}(v)=t$ for all $v \in R$, where $R = V(N)^{c}.$  Moreover, $s(2k)=(n-2k)t$, $\frac{n\left( m-k\right) }{2k\left( n-2k\right) }=\mu _{1},$ and $\delta =
\mathit{a}\left( G\right) .$
\end{corollary}

\begin{proof}
Since $N$ is a DIM, $R = V(N)^{c}$ is an independent set of $G$ with $n-2k$ vertices and $\left\vert E_{R}\left( G\right) \right\vert =\sum_{v \in R} d\left( v \right)= m-k.$ The result follows from Corollary~\ref{independence}.
\end{proof}

\begin{remark}
Since the Laplacian eigenvalues are either integers or irrational numbers,
see e. g. in \cite{Fallat_Kirk}, the equality case in Corollary \ref{DIM2} implies that either $n$, $m-k$ or both
must be even numbers.
\end{remark}

\begin{corollary}
Let $G$ be a graph with $n$ vertices. Let $N=\left\{ \ i_{1}j_{1},\ldots
,i_{k}j_{k}\right\} $ be a DIM of $G$, where $2k\leq n-1.$ Let $ X= V\left( N\right) $. Then
\begin{equation*}
S_{L}(G)\geq \alpha _{X}-1.
\end{equation*}%
If equality holds then the conditions 1. and 2. of Proposition \ref{Lem1} hold for $X=V(N).$
Moreover, $\rho _{G}\left( V\left( N\right) \right) =\mu _{1}$ and $\frac{1}{%
n-2k}\sum_{v\in V\left( N\right) ^{c}}d\left( v\right) =\mathit{a}\left(
G\right) .$
\end{corollary}

\begin{proof}
Let $X=V\left( N\right) $ and $n_{1}=2k$.\
Taking into account Remark \ref{Important},
\begin{eqnarray*}
\alpha _{X}  & = & \frac{1}{2k}\sum_{v\in V(N)}d\left( v\right), \\
\alpha _{X^{c}}  & = & \frac{1}{n-2k}\sum_{v\in V(N)^{c}}d\left( v\right),\\
2m_{X}        & = & \sum_{v\in V(N)}d_{V(N)}\left( v\right) =2k \mbox{, and}\\
 2m_{X^{c}}       & = & \sum_{v\in V(N)^{c}}d_{V(N)^{c}}\left( v\right) =0.
\end{eqnarray*}
Therefore, $\rho _{G}\left( V(N)\right)
=\alpha _{X}-\frac{2k}{2k}+$ $\alpha _{X^{c}}$.
We now apply Corollary \ref{new} to obtain
\begin{equation*}
S_{L}(G)\geq \left\vert \rho _{G}\left( X\right) -\alpha _{X^{c}}\right\vert
=\alpha _{X}-1.
\end{equation*}%
By Corollary \ref{Theorem1}, if equality holds there exists an
integer $s$ such that $d_{V(N)^{c}}\left( v\right) =s,\ $ for all $v\in V(N),$
$d_{V(N)}\left( v\right) =t$ for all $v\in V(N)^{c}$, and $s(2k)=t(n-2k)$. Additionally, $\rho _{G}\left(
V(N)\right) =\alpha _{X}-1+$ $\alpha _{X^{c}}=\mu _{1}$,  and $\mathit{a}\left( G\right) = \alpha _{X^{c}}.$
\end{proof}

\subsection{ Lower bounds for the Laplacian spread of graphs with $(\kappa ,\tau )$-regular subset of vertices}

Now we search for lower bounds for $S_{L}(G)$ in terms of the order of $S$,
where $S$ is a $(\kappa ,\tau )$-regular subset of $V$.



\begin{corollary}
Let $G=\left( V,E\right) $ be a graph with $n$ vertices and $S$ a $(\kappa, \tau)$- regular subset of $V$ with cardinality $n_1,$  and suppose that there
exists a vertex $v$ $\in S^{c}$ with no neighbors in $S^{c}$. Then%
\begin{equation}
S_{L}(G)\geq \frac{\tau \left( n-n_{1}\right) }{n_{1}}.
\end{equation}%
If the equality holds there exists an integer $s$ such that $d_{S^{c}}\left(
v\right) =s,\ $ for all $v\in S$, $sn_{1}=\tau(n-n_{1})$, with  $\mu _{1}=\frac{\tau n}{n_{1}}$ and $%
\mathit{a}\left( G\right) =\tau .$
\end{corollary}

\begin{proof}
Let $S$ be a $(\kappa, \tau)$- regular subset of $V$ with cardinality $n_1,$ $\left\vert E_{S^{c}}\left( G\right) \right\vert = (n-n_1)\tau$ and
\begin{equation*}
\rho _{G}\left( S\right) =\rho _{G}\left( S^{c}\right) =\frac{n\left(n-n_{1}\right) \tau }{n_{1}\left( n-n_{1}\right) }=\frac{\tau n}{n_{1}}.
\end{equation*}
Let $v\in S^{c}$ be a vertex  with no neighbors  in $S^{c}.$ Then, by
removing the $\tau $ edges connecting $v$ with its $\tau $ neighbors in $S,$
the graph $G$ becomes disconnected. In consequence, $\mathit{a}\left(
G\right) \leq \tau \leq \frac{\tau n}{n_{1}}\leq \mu _{1}$ and 
\begin{equation*}
S_{L}(G)\geq \frac{\tau n}{n_{1}}-\tau =\frac{\tau \left( n-n_{1}\right) }{%
n_{1}}.
\end{equation*}%
The equality case follows from the equality case in Corollary \ref{Theorem1}
\end{proof}

\begin{corollary}
Let $G=\left( V,E\right) $ be a graph with $n$ vertices, $S$ a $(\kappa, \tau)$- regular subset of $V$ with cardinality $n_1$ and $\kappa \leq \tau $.
Suppose that there exists a vertex $v$ $\in S$ with no neighbors in $S^{c}$.\
Then%
\begin{equation}
S_{L}(G)>\frac{\tau n}{n_{1}}-\kappa .
\end{equation}
\end{corollary}

\begin{proof}
Let $v\in S$ in the condition of the statement then, removing $%
\kappa $ edges connecting $v$ with its $\kappa $ neighbors in $S$ the graph $%
G$ becomes disconnected. Therefore $\mathit{a}\left( G\right) \leq
\kappa \leq \tau \leq \frac{\tau n}{n_{1}}\leq \mu _{1}$. Taking $
\phi =\kappa $ in Corollary \ref{Theorem1} we obtain
\begin{equation*}
S_{L}(G)\geq \frac{\tau n}{n_{1}}-\kappa,
\end{equation*}%
and, as $d_{S^c}(v)=0$, there is no equality. Then, the result follows.
\end{proof}

\begin{corollary}
\label{difference}Let $G=\left( V,E\right) $ be a graph with $n$ vertices and $S$ a $(\kappa, \tau)$- regular subset of $V$ with cardinality $n_1.$  If $
\tau -\kappa \leq 1$, then
\begin{equation}
S_{L}(G)\geq \Delta +1-\frac{\tau n}{n_{1}}.  \label{low}
\end{equation}%
If the equality holds there exists an integer $s$ such that $d_{S^{c}}\left(
v\right) =s$, for all $v\in S$ and $sn_{1}=\tau(n-n_{1})$. Moreover,\ $\mu _{1}=\Delta +1$ and $a\left( G\right) =\frac{\tau n}{n_{1}}.$
\end{corollary}

\begin{proof}
Since $S$ is a $(\kappa, \tau)$- regular subset of $V$ with cardinality $n_1$ and $\tau -\kappa \leq 1$, $\left\vert E_{S}\left( G\right) \right\vert = \sum_{v \in S} d(v) - \sum_{ v\in S} d_{S}(v) \leq n_1(\Delta - \kappa) \leq n_1(\Delta +1 - \tau)$. It follows from Proposition \ref{Lem1} that
$$\mathit{a}\left( G\right) \leq \frac{\tau n}{n_1} = \tau - \frac{(n-n_1)\tau}{n_1} = \tau + \frac{\left\vert E_{S}\left( G\right) \right\vert}{n_1}\leq \Delta +1 \leq \mu_1$$
Taking $\phi =\Delta +1$ in Corollary \ref{Theorem1}, the result follows.
\end{proof}
\begin{corollary}\label{special_2}
Let $G=\left( V,E\right) $ be a graph with $n$ vertices and $S$ a $(\kappa, \tau)$- regular subset of $V$ with cardinality $n_1.$ Then
\begin{equation}
S_{L}(G)\geq \left\vert \kappa -\tau \right\vert .
\end{equation}
If equality holds there exists an integer $s,$ $d_{S^{c}}\left( v\right) =s$
for all $v\in S$ and $sn_{1}=\tau(n-n_{1})$.
\end{corollary}

\begin{proof}
As $S$ is a $(\kappa, \tau)$- regular subset of $V$ with cardinality $n_1, $
\begin{eqnarray*}
\alpha _{S}=\frac{1}{n_{1}}\sum_{v\in
S}d\left( v\right) = \frac{1}{n_1}\left( \sum_{v \in S} d_{S}(v) + \left\vert E_{S}\left( G\right) \right\vert \right)\\ = \frac{1}{n_1}\left(n_1 \kappa + \left(n - n_1\right) \tau \right) = \kappa + \frac{n-n_1}{n_1} \tau
\\= \kappa -\tau + \frac{n\tau}{n_1} = \kappa - \tau + \rho_{G}\left( S \right).
\end{eqnarray*}

Applying Corollary \ref{new} we obtain
\begin{equation*}
S_{L}(G)\geq \left\vert \rho _{G}\left( S \right) -\alpha _{S}\right\vert
=\left\vert \kappa -\tau \right\vert .
\end{equation*}%
Moreover, by Corollary \ref{Theorem1}, if equality holds there exists an
integer $s,$ $d_{S^{c}}\left( v\right) =s$ for all $v\in S$, and $s n_1 = \tau (n-n_1)$.
\end{proof}

\begin{example}
If $S=V\left( C_{4}\right) $ in $G=C_{4}\vee C_{4}$ then $S$ is a $\left(
2,4\right) =\left( \kappa ,\tau \right) \ $-regular set of $V\left( G\right) .$
Moreover, $\Delta =\delta =6$, and therefore, $\Delta +1-\delta
=1<\left\vert -\kappa +\tau \right\vert =2=S_{L}(G)$. Thus, in this case,
the lower bound in Corollary \ref{special_2} is sharp and it is an
improvement of (\ref{basic_lower_bound}).
\end{example}

\section{Lower bound for Laplacian spread of graphs with prescribed degree sequences}

In this section, using some known numerical inequalities and an appropriate parameter, we present lower bounds for the Laplacian spread of a graph with a prescribed degree sequence.

We start recalling  the following numerical inequality, \cite{Pecaric}.

\begin{lemma}
\label{numineq}Let $a=\left( a_{1},\ldots ,a_{n}\right) $ and $b=\left(
b_{1},\ldots ,b_{n}\right) $ be two positive $n$-vectors with $0<m_{1}\leq
a_{i}\leq M_{1}$ and $0<m_{2}\leq b_{i}\leq M_{2}$, $i\in \{1,\ldots ,n\},$  and constants $m_{1},m_{2},M_{1},  M_{2}$. The following complement of Cauchy's inequality holds
\begin{equation}
\frac{\sum_{i=1}^{n}a_{i}^{2}}{\sum_{i=1}^{n}a_{i}b_{i}}-%
\frac{\sum_{i=1}^{n}a_{i}b_{i}}{\sum_{i=1}^{n}b_{i}^{2}}\leq
\left( \left( \frac{M_{1}}{m_{2}}\right) ^{\frac{1}{2}}-\left( \frac{m_{1}}{%
M_{2}}\right) ^{\frac{1}{2}}\right) ^{2}.  \label{ineqnume}
\end{equation}
\end{lemma}

Recall that the first Zagreb index of a graph is defined as $Z_g(G) =\sum\limits_{\substack{i=1}}^{n}d_{i}^{2}$. So, using this invariant we have the next lemma.

\begin{lemma}
\label{grapheq}Let $\mu _{1}\geq \mu _{2}\geq \cdots \geq \mu _{n-1}\geq \mu
_{n}=0$ be the Laplacian eigenvalues of a graph $G$ with $m$ edges and
degrees sequence $d_{1},d_{2},\ldots, d_{n}$. For $i\neq j$ let us denote by $%
\varkappa _{ji}$ the number of vertices at the intersection of the
neighborhoods of $v_{i}$ and $v_{j}$. Then,

\begin{enumerate}

\item $\sum\limits_{\substack{i=1}}^{n-1}\mu_{i}^{2}=Z_g(G)+2m;$\\

\item $ \sum\limits_{\substack{i=1}}^{n-1}\mu _{i}^{4}  =\Upsilon $

\end{enumerate}
where

\begin{equation*}
\Upsilon =\sum \limits_{i=1}^{n}\left(
d_{i}^{4}+3d_{i}^{3}\right)+ Z_{g}(G) +\sum\limits_{i=1}^{n}\left(
\sum\limits_{\substack{s=1\\s\neq i}}^{n}\varkappa _{is}^{2}+\sum\limits_{v_{s}\sim
v_{i}}2d_{i}\left( d_{s}-\varkappa _{si}\right) +d_{s}\left(
d_{s}-2\varkappa _{si}\right)\right).
\end{equation*}

\end{lemma}

\begin{proof}
The equality in $1$. is obtained from the Frobenius matrix norm
computation of $L\left( G\right) $. For the second
equality note that $\digamma=\sum_{\substack{i=1}}^{n-1}\mu _{i}^{4}=trace\left(
L\left( G\right) ^{4}\right) .$
Let $e_{i}$ be the $i$-th canonical vector of $\mathbb{R}^{n}$. Therefore,
\begin{equation*}
\digamma
=\sum\limits_{\substack{i=1}}^{n}e_{i}^{t}L\left( G\right)
^{4}e_{i}=\sum\limits_{\substack{i=1}}^{n}\left\Vert L\left( G\right)
^{2}e_{i}\right\Vert ^{2},
\end{equation*}
where $\left\Vert \cdot \right\Vert $ stands for the
Euclidean vector norm. Setting $L\left( G\right) =\left( \delta _{ij}\right), $
then
\begin{equation*}
\delta _{ij}=\left\{
\begin{tabular}{ll}
$d_{i}$ &  if\ $i=j$ \\
$-1$ & if $v_{i}\sim v_{j}$%
\end{tabular}%
\right.
.\end{equation*}%
Since,
$$L\left( G\right) e_{i}=\left( \delta _{1i},\ldots ,\delta
_{ni}\right) ^{t},$$ we have
{\footnotesize
\begin{eqnarray*}
L\left( G\right) ^{2}e_{i}&= & \left( \sum\limits_{\substack{t=1}}^{n}\delta _{t1}\delta _{ti},\ldots
,\sum\limits_{\substack{t=1}}^{n}\delta _{ti}\delta _{ti},\ldots
,\sum\limits_{\substack{t=1}}^{n}\delta _{tn}\delta _{ti}\right) ^{t} \\
&=&\left( \left( d_{1}+d_{i}\right) \delta _{1i}+\sum\limits_{\substack{t=2 \\ t\neq
i}}^{n}\delta _{t1}\delta _{ti},\ldots ,d_{i}+d_{i}^{2},\ldots ,\left(
d_{n}+d_{i}\right) \delta _{ni}+\sum\limits_{\substack{t=1 \\ t\neq i,n}}^{n}\delta
_{t1}\delta _{ti}\right) ^{t} \\
&=&\left( \left( d_{1}+d_{i}\right) \delta _{1i}+\varkappa _{1i},\ldots
,d_{i}+d_{i}^{2},\ldots ,\left( d_{n}+d_{i}\right) \delta _{ni}+\varkappa
_{ni}\right) ^{t}.
\end{eqnarray*}
}
Therefore,
{\footnotesize
\begin{eqnarray*}
\left\Vert L\left( G\right) ^{2}e_{i}\right\Vert ^{2}&=& \left(
d_{i}+d_{i}^{2}\right) ^{2}+\sum\limits_{\substack{s=1 \\ s\neq i}}^{n}\left( \left(
d_{s}+d_{i}\right) \delta _{si}+\varkappa _{si}\right) ^{2}\\
&=& \left(d_{i}+d_{i}^{2}\right) ^{2}+\sum\limits_{\substack{s=1 \\s\neq i}}^{n}\left( \varkappa
_{is}^{2}+\delta _{si}^{2}\left( d_{i}+d_{s}\right) ^{2}+2\left(
d_{s}+d_{i}\right) \delta _{si}\varkappa _{si}\right) \\
&=&\left( d_{i}+d_{i}^{2}\right) ^{2}+\sum\limits_{\substack{s=1 \\ s\neq i}}^{n}\varkappa
_{is}^{2}+\sum\limits_{\substack{s=1 \\ s\neq i}}^{n}\delta _{si}^{2}\left(
d_{i}^{2}+2d_{i}d_{s}+d_{s}^{2}\right) +2\sum\limits_{\substack{ s=1\\ s\neq
i}}^{n}\left( d_{s}+d_{i}\right) \delta _{si}\varkappa _{si}\\
&=&\left( d_{i}+d_{i}^{2}\right) ^{2}+\sum\limits_{\substack{s=1 \\ s\neq i}}^{n}\varkappa
_{is}^{2}+d_{i}^{3}+2d_{i}\sum\limits_{\substack{ v_{s}\sim
v_{i}}}d_{s}+\sum\limits_{\substack{ v_{s}\sim v_{i}}}d_{s}^{2}-2\sum\limits_{\substack{v_{s}\sim
v_{i}}}d_{s}\varkappa _{si}-2d_{i}\sum\limits_{\substack{v_{s}\sim v_{i}}}\varkappa
_{si}\\
&=&d_{i}^{2}+3d_{i}^{3}+d_{i}^{4}+\sum\limits_{\substack{s=1 \\ s\neq i}}^{n}\varkappa
_{is}^{2}+2d_{i}\sum\limits_{\substack{ v_{s}\sim v_{i}}}\left( d_{s}-\varkappa
_{si}\right) +\sum\limits_{\substack{v_{s}\sim
v_{i}}}d_{s}^{2}-2\sum\limits_{\substack{v_{s}\sim v_{i}}}d_{s}\varkappa _{si}.
\end{eqnarray*}
}
\end{proof}

As a consequence of the previous two lemmas we have the next result.

\begin{theorem} \label{first}
Let $\mu _{1}\geq \mu _{2}\geq \cdots \geq \mu _{n-1}\geq \mu _{n}=0$ be the
Laplacian eigenvalues of a graph $G$ with $m$ edges and degrees sequence $
d_{1},d_{2},\ldots, d_{n}$. For $i\neq j$, let $\varkappa _{ji}$
the number of vertices at the intersection of the neighborhoods of $v_{i}$
and $v_{j}$.
Then \begin{equation*}
\left(\frac{\Upsilon (n-1)-(Z_g(G)+2m)^{2}}{(n-1) (Z_g(G)+2m)}\right)^{\tfrac{1}{2}}\leq S_{L}(G).
\end{equation*}

\end{theorem}

\begin{proof}
In this proof we use Lemma \ref{numineq} replacing $a_{i}=\mu _{i}^{2}$
and $b_{i}=1$, for $i=1,\ldots ,n-1$. Note that,
\begin{equation*}
\mu _{n-1}^{2}\leq \mu _{i}^{2}\leq \mu _{1}^{2}\quad \left( i=1,\ldots
,n-1\right).
\end{equation*}
Moreover, $m_{1}=\mu _{n-1}^{2}$, $M_{1}=\mu _{1}^{2}$,\ and $m_{2}=1=M_{2}.$
Replacing in the inequality (\ref{ineqnume}) we obtain
\begin{equation*}
\frac{\sum\limits_{\substack{i=1}}^{n-1}\mu_{i}^{4}}{\sum\limits_{\substack{i=1}}^{n-1}\mu_{i}^{2}}-\frac{\sum\limits_{\substack{i=1}}^{n-1}\mu _{i}^{2}}{n-1}\leq \left( \left(
\mu _{1}^{2}\right)^{\frac{1}{2}}-\left(\mu _{n-1}^{2}\right) ^{\frac{1}{2}%
}\right) ^{2}.
\end{equation*}
Using identities $1.$ and $2.$ in Lemma \ref{grapheq} the result follows.
\end{proof}

Another interesting numerical inequality was presented in  \cite{Pachpatte}.

\begin{lemma}  \cite{Pachpatte}
\label{pachpa}Let $a\leq a_{i}\leq A$ and $b\leq b_{i}\leq B$, $\left(
i=1,\ldots ,n\right) $. Moreover consider, $t_{i}\geq 0$, $\left( i=1,\ldots
,n\right) $ and $T=\sum\limits_{\substack{i=1}}^{n}t_{i}$. Then
\begin{equation}
\left\vert \frac{1}{T}\sum\limits_{\substack{i=1}}^{n}t_{i}a_{i}b_{i}-\frac{1}{T^{2}}%
\sum\limits_{\substack{i=1}}^{n}t_{i}a_{i}\sum\limits_{\substack{i=1}}^{n}t_{i}b_{i}\right\vert
\leq \frac{1}{4}\left( A-a\right) \left( B-b\right) .  \label{numerique2}
\end{equation}
\end{lemma}
Then, applying this result we can obtain the next lower bound for the Laplacian spread of a graph with a prescribed degree sequence.
\begin{theorem}
\label{Th2}Let $\mu _{1}\geq \mu _{2}\geq \cdots \geq \mu _{n-1}\geq \mu
_{n}=0$ be the Laplacian eigenvalues of a graph $G$ with $m$ edges and
degrees sequence $d_{1},d_{2},\ldots, d_{n}$. For $i\neq j$, let us denote by $%
\varkappa _{ji}$ the number of vertices at the intersection of the
neighborhoods of $v_{i}$ and $v_{j}$. Then%
\begin{equation*}
\frac{2}{n-1}\left( \left( n-1\right) \left(
Z_{g}(G)+2m\right) -4m^{2}\right) ^{\frac{1}{2}}\leq
S_{L}(G).
\end{equation*}
\end{theorem}

\begin{proof}
In this proof we use Lemma \ref{pachpa} replacing $a_{i}=b_{i}=\mu _{i}$,
$\left( i=1,\ldots ,n-1\right) $. Moreover, $t_{i}=\frac{1}{n-1},$ $\left(
i=1,\ldots ,n-1\right) $.\ Since,%
\begin{equation*}
\mu _{n-1}\leq \mu _{i}\leq \mu _{1}\quad \left( i=1,\ldots ,n-1\right),
\end{equation*}%
we have, $a=b=\mu _{n-1}$ and $A=B=\mu _{1}$. Therefore, from inequality (\ref{numerique2}) we obtain
\begin{equation*}
\left\vert \frac{1}{n-1}\sum\limits_{\substack{i=1}}^{n-1}\mu _{i}^{2}-\frac{1}{\left(
n-1\right) ^{2}}\left( \sum\limits_{\substack{i=1}}^{n-1}\mu _{i}\right)
^{2}\right\vert \leq \frac{1}{4}\left( \mu _{1}-\mu _{n-1}\right) ^{2}.
\end{equation*}%
From identity $1.$ in Lemma \ref{grapheq},
\begin{equation*}
\frac{2}{n-1}\left( \left( n-1\right) \left(
Z_g(G)+2m\right) -4m^{2}\right) ^{\frac{1}{2}}\leq
\mu _{1}-\mu _{n-1}
\end{equation*}%
and the result follows.
\end{proof}

We observe that the above result was also obtained in \cite{Liu}.

Let us denote by $\mathbb{J}_{p}$ the all ones matrix of order $p.$

The next theorem, due to Brauer, relates the eigenvalues of an arbitrary matrix
and the matrix resulting from it after a rank one additive perturbation.

\begin{theorem}
\cite{Brauer} \label{ber copy(1)} Let $M$ be an arbitrary $n\times n$ matrix
with eigenvalues $\theta _{1},\ldots ,\theta _{n}$. Let $\mathbf{x}_{k}$ be
an eigenvector of $M$ associated with the eigenvalue $\theta _{k}$, and let $
\mathbf{q}$ be an arbitrary $n$-dimensional vector. Then the matrix $M+
\mathbf{x}_{k}\mathbf{q}^{T}$ has eigenvalues $\theta _{1},\ldots ,\theta
_{k-1},\theta _{k}+\mathbf{x}_{k}^{T}\mathbf{q},\theta _{k+1},\ldots ,\theta
_{n}.$
\end{theorem}
Let $\phi$ be a real number and consider the matrix
\begin{equation}
M=L(G)+\frac{\phi }{n}\mathbb{J}_{n},  \label{mM}
\end{equation}
obtained after applying Theorem \ref{ber copy(1)} replacing $M$ by
$L\left( G\right) $ and the eigenpair $\left( \theta _{k},\mathbf{q}\right) $
of $M$ with the eigenpair $\left( 0,\mathbf{e}\right) $ of $L\left( G\right)
$. Thus, the matrix $L(G)+\frac{\phi }{n}\mathbb{J}_{n}$ has eigenvalues $
\mu _{1},\ldots ,\mu _{k-1},0+\frac{\phi }{n}\mathbf{e}^{t}\mathbf{e},\mu
_{k+1},\ldots ,\mu _{n-1}.~\ $
Then, if $\mu _{1}\geq \phi \geq \mu _{n-1}$ the matricial spread of $M$, $S\left( M\right) $ and the Laplacian spread of
$G$,\ coincide.

Applying Lemma \ref{numineq} to the spectrum of the matrix $M$ in (\ref{mM}) we get
the following result.

\begin{theorem}
\label{Th1_1}Let $\mu _{1}\geq \mu _{2}\geq \ldots \geq \mu _{n-1}\geq \mu
_{n}=0$ be the Laplacian eigenvalues of a graph $G$ with $m$ edges and
degrees sequence $d_{1},d_{2},\ldots d_{n}$. Let $\mu _{n-1}\leq \phi \leq
\mu _{1}.$\ \ For $i\neq j$, let us denote by $\varkappa _{ji}$ the number of
vertices at the intersection of the neighborhoods of $v_{i}$ and $v_{j}$.
Then

\begin{equation*}
\left(\frac{n(\Upsilon +\phi ^{4})-(Z_{g}(G)+2m+\phi ^{2})^2 }{n(Z_{g}(G)+2m+\phi ^{2})} \right)^{\frac{1}{2}}\leq S_{L}(G).
\end{equation*}

\end{theorem}

\begin{proof}
In this proof we use Lemma \ref{numineq} replacing $a_{i}=\mu _{i}^{2},$ $b_{i}=1$, for $i=1,\ldots ,n-1$ and setting $a_{n}=\phi ,$ as $\mu
_{n-1}\leq \phi \leq \mu _{1}$.\ Note that,%
\begin{equation*}
\mu _{n-1}^{2}\leq \phi^{2}\leq \mu _{1}^{2}\quad \left( i=1,\ldots
,n\right).
\end{equation*}
Moreover, $m_{1}=\mu _{n-1}^{2}$, $M_{1}=\mu _{1}^{2}$,\ and $m_{2}=1=M_{2}$%
. Replacing these values in inequality (\ref{ineqnume}) we have
\begin{equation*}
\frac{\phi ^{4}+\sum\limits_{\substack{i=1}}^{n-1}\mu _{i}^{4}}{\phi
^{2}+\sum\limits_{\substack{i=1}}^{n-1}\mu _{i}^{2}}-\frac{\phi
^{2}+\sum\limits_{\substack{i=1}}^{n-1}\mu _{i}^{2}}{n}\leq \left( \left( \mu
_{1}^{2}\right) ^{\frac{1}{2}}-\left( \mu _{n-1}^{2}\right) ^{\frac{1}{2}%
}\right) ^{2}.
\end{equation*}%
Usiing identities $1.\ $and $2$. in Lemma \ref{grapheq} the result follows.
\end{proof}

Applying Lemma \ref{pachpa} to the matrix $M$ in (\ref{mM}) we get the following
inequality.

\begin{theorem}
\label{Th1_2}Let $\mu _{1}\geq \mu _{2}\geq \ldots \geq \mu _{n-1}\geq \mu
_{n}=0$ be the Laplacian eigenvalues of a graph $G$ with $m$ edges and
degrees sequence $d_{1},d_{2},\ldots d_{n}$. Let $\mu _{n-1}\leq \phi \leq
\mu _{1}.\ $Then%
\begin{equation*}
\frac{2}{n}\left\vert n Z_{g}(G)+2mn+\left( n-1\right)
\phi ^{2}-4m\left( m+\phi \right) \right\vert ^{\frac{1}{2}}\leq S_{L}(G).
\end{equation*}
\end{theorem}

\begin{proof}
In this proof we use Lemma \ref{pachpa} replacing $a_{i}=b_{i}=\mu _{i}$,
$\left( i=1,\ldots ,n-1\right) $ and $a_{n}=b_{n}=\phi $. Moreover, $t_{i}=%
\frac{1}{n},$ $\left( i=1,\ldots ,n\right) $. Since,%
\begin{equation*}
\mu _{n-1}\leq a_{i}\leq \mu _{1}\quad \left( i=1,\ldots ,n\right)
\end{equation*}%
we have, $a=b=\mu _{n-1}$ and $A=B=\mu _{1}$. Replacing in the
inequality (\ref{numerique2}) we obtain
\begin{equation*}
\left\vert \frac{1}{n}\left( \sum\limits_{\substack{i=1}}^{n-1}\mu _{i}^{2}+\phi
^{2}\right) -\frac{1}{n^{2}}\left( \sum\limits_{\substack{i=1}}^{n-1}\mu _{i}+\phi
\right) ^{2}\right\vert \leq \frac{1}{4}\left( \mu _{1}-\mu _{n-1}\right)
^{2}.
\end{equation*}%
Using identity $1.\ $in Lemma \ref{grapheq}
\begin{equation*}
\frac{4}{n^{2}}\left\vert n\left( 2m+Z_{g}(G)+\phi
^{2}\right) -\left( 2m+\phi \right) ^{2}\right\vert \leq \left( \mu
_{1}-\mu _{n-1}\right) ^{2}
\end{equation*}%
$\Rightarrow $%
\begin{equation*}
\frac{4}{n^{2}}\left\vert n\left( 2m+Z_{g}(G)+\phi
^{2}\right) -4m^{2}-4m\phi -\phi ^{2}\right\vert \leq \left( \mu
_{1}-\mu _{n-1}\right) ^{2}
\end{equation*}%
$\Rightarrow $%
\begin{equation*}
\frac{4}{n^{2}}\left\vert n Z_{g}(G)+2mn+\left(
n-1\right) \phi ^{2}-4m\left( m+\phi \right) \right\vert \leq \left( \mu
_{1}-\mu _{n-1}\right) ^{2}.
\end{equation*}
\end{proof}

\section{An Example}
In the next example, a computational experiment is presented in order to approach the Laplacian spread of $G$. In each column, we use different parameters, namely:
\begin{enumerate}
\item algebraic connectivity;
\item edge density of a set $X$;
\item minimum vertex degree,
\end{enumerate}
and we test the lower bounds presented at Corollary \ref{independence copy(1)} , Theorem  \ref{first},  Theorem \ref{Th2}, Theorem \ref{Th1_1} and  Theorem \ref{Th1_2}.

 In this example, it can be seen that the edge density gives a better approach to algebraic connectivity compared with the minimum vertex degree $\delta$. In consequence, better results are obtained in the final lower bounds for the Laplacian spread when we replace the parameter $\phi$ by the edge density rather than the minimum degree.

\begin{example}
Let $X=\left\{ 1,2,3\right\} \subseteq V\left( G\right) $, where $G$ is the
graph at Figure 1. For $G$ we have $a\left( G\right) =1.75,\ \delta
\left( G\right) =3,\ \rho _{G}\left( X\right) =2,\ S_{L}\left( G\right) =5.8$.

\begin{figure}

\begin{center}
 \begin{tikzpicture}[scale=0.4]


 \draw  (-3,4.5) node{$ G$:};
  \draw[fill]  (0,2) circle (0.3cm);
  \draw[fill]  (0,-2) circle (0.3cm);
  \draw[fill]  (2,0) circle (0.3cm);
 \draw[fill]  (14,4) circle (0.3cm);
 \draw[fill]  (10,2) circle (0.3cm);
 \draw[fill]  (18,2) circle (0.3cm);
 \draw[fill]  (10,-1) circle (0.3cm);
 \draw[fill]  (18,-1) circle (0.3cm);

 \draw[fill]  (14,-4) circle (0.3cm);

 \draw  (0,3) node{$1$};
 \draw (2,1) node{$2$};
  \draw (0,-3) node{$3$};
  \draw (9,2) node{$9$};
 \draw (19,2) node{$5$};
 \draw  (9,0) node{$8$};
 \draw (19,0) node{$6$};
  \draw  (14,5) node{$4$};
  \draw  (14,-5) node{$7$};
%
\draw (10,2) --(10,-1);
\draw (18,2) --(18,-1);
\draw (10,2)--(14,4);
\draw(10,-1)--(14,-4);
\draw(18,-1)--(14,-4);
\draw(18,2)--(14,4);
\draw (10,2) --(18,2);
\draw (10,-1) --(18,-1);
\draw (10,2) --(18,-1);
\draw (18,2) --(10,-1);
\draw (14,4) --(14,-4);

\draw (14,4) --(10,-1);
\draw (14,4) --(18,-1);

\draw (14,-4) --(10,2);
\draw (14,-4) --(18,2);

\draw (0,2) --(0,-2);
\draw (2,0) --(0,-2);
\draw (0,2) --(2,0);

\draw (0,2) --(14,4);
\draw (0,2) --(10,2);
\draw (2,0) --(18,2);
\draw (0,-2) --(14,-4);

 \end{tikzpicture}

  \end{center}
 \caption{Graph $G$}
  \label{ex:example1 Graph $G$}

 \end{figure}

 \begin{equation*}
\begin{tabular}{|c|c|c|c|}
\hline
$\phi $ & $a\left( G\right) $ & $\rho _{G}\left( X\right) $ & $\delta \left(
G\right) $ \\ \hline
{\tiny Corollary \ref{independence copy(1)}} & $2$ & $2$ & $2$ \\ \hline
{\tiny Th \ref{first}} & $3.2313$ & $3.2313$ & $3.2313$ \\ \hline
{\tiny Th \ref{Th1_1}} & $3.6636$ & $3.6316$ & $3.4762$ \\ \hline
{\tiny Th \ref{Th2}} & $3.8730$ & $3.8730$ & $3.8730$ \\ \hline
{\tiny Th \ref{Th1_2}} & $4.3477$ & $4.2630$ & $3.9752$ \\ \hline
\end{tabular}%
\end{equation*}
\end{example}

 \newpage

\textbf{Acknowledgments}. Enide Andrade was supported in part by the Portuguese Foundation for Science and Technology (FCT-Funda\c c\~ao para a Ci\^encia e a Tecnologia), through CIDMA - Center for Research and Development in Mathematics and Applications, within project UID/MAT/04106/2013.
M. Robbiano was supported by
project VRIDT-UCN 2014 N. 220202-10301403 and J. Rodr\'{\i}guez was
supported by CONICYT-PCHA/Doctorado Nacional/2015-21150477. M.A.A. de Freitas was partially suported by CNPQ - Brazilian Council for Scientific and Technological Development (grant 308811/2014-3)  and by FAPERJ - Foundation for Support of Research in
the State of Rio de Janeiro (grant E-26/201.536/2014).

\end{document}